\documentclass[11pt]{amsart}
\newcommand{\myemail}[1]{\href{mailto:#1}{#1}}
\title{3-braid knots with maximal  4-genus}
\author[Sebastian Baader]{S.~Baader}
\address{Mathematisches Institut, Sidlerstrasse~5, 3012 Bern, Switzerland}
\email{\myemail{sebastian.baader@unibe.ch}}
\author[Lukas Lewark]{L.~Lewark}
\address{Universit\"at Regensburg, 93053 Regensburg, Germany}
\email{\myemail{lukas@lewark.de}}
\urladdr{\url{www.lewark.de/lukas}}
\author[Filip Misev]{F.~Misev}
\address{Universit\"at Regensburg, 93053 Regensburg, Germany}
\email{\myemail{filip.misev@mathematik.uni-regensburg.de}}
\author[Paula Tru\"ol]{P.~Tru\"ol}
\address{ETH Z\"urich, R\"amistrasse 101, 8092 Z\"urich, Switzerland}
\email{\myemail{paulagtruoel@gmail.com}}
\urladdr{\url{https://people.math.ethz.ch/~ptruoel/}}
\usepackage[utf8]{inputenc}
\usepackage[final]{microtype}
\usepackage{amssymb,amsmath,caption,array,multirow,pifont,amsfonts,amsthm,bbm,graphicx,psfrag,pstricks,mathrsfs,pifont,enumitem,pdfpages,xcolor,thmtools,thm-restate,picinpar,mathtools,tabto}
\usepackage[bookmarks=true,hidelinks]{hyperref}
\usepackage[nameinlink]{cleveref}
\let\cref\Cref
\declaretheorem[numberwithin=section,name=Theorem]{theorem}
\newtheorem{corollary}[theorem]{Corollary}
\newtheorem{proposition}[theorem]{Proposition}
\newtheorem{lemma}[theorem]{Lemma}
\newtheorem{definition}[theorem]{Definition}
\newtheorem{question}[theorem]{Question}

\theoremstyle{remark}
\newtheorem{example}[theorem]{Example}
\newtheorem{remark}[theorem]{Remark}
\def\geq{\geqslant}
\def\leq{\leqslant}
\def\x{\times}
\DeclareMathOperator{\tu}{tu}
\def\gtop{g_4^{\text{top}}}

\def\Z{{\mathbb Z}}

\def\R{{\mathbb R}}
\DeclareMathOperator{\cl}{cl}

\DeclareMathOperator{\rev}{rev}
\newcommand{\hatsigma}{\widehat{\sigma}}
\def\zeta3{e^{2\pi i/3}}
\newcommand{\twist}{\rightarrowtail}

\allowdisplaybreaks
\begin{document}
\begin{abstract}
We classify 3-braid knots whose topological 4-genus coincides with their Seifert genus, using McCoy's twisting method and the Xu normal form. In addition, we give upper bounds for the topological 4-genus of positive and strongly quasipositive 3-braid knots.
\end{abstract}
\maketitle
\section{Introduction} 
Four decades after Freedman's celebrated work on 4-manifolds~\cite{F}, the topological 4-genus $\gtop(K)$ of knots~$K$ remains difficult to determine. The first challenge is posed by the figure-eight knot $4_1$,
which satisfies $\gtop(4_1)=1$, although 
its second power $4_1 \# 4_1$ bounds an embedded disc in the 4-ball.
This causes all its additive lower bounds on the 4-genus to be
trivial. In particular, the signature invariant $\sigma(4_1)$ is zero. Due to this example, there is little reason to believe that the inequality
$$|\sigma(K)| \leq 2\gtop(K)$$
has much to tell us about the (topological) 4-genus of knots in general. In this note, we show that closures of 3-braids are exceptional in this respect. More precisely, we will see that the figure-eight knot is
exceptional among closed 3-braids in that it is the only 3-braid knot~$K$ that satisfies $|\sigma(K)|<2g(K)$, yet $\gtop(K)=g(K)$, where $g(K)$ denotes the ordinary Seifert genus.

\begin{theorem} \label{thm:1} \label{theorem1}
Let $K$ be a $3$-braid knot other than the figure-eight knot. Then
\[ |\sigma(K)| = 2g(K) \quad\Longleftrightarrow\quad \gtop(K)=g(K). \]
These equalities hold precisely if $K$ or its mirror is one of the following knots:
\begin{itemize}
\item $T(2,2m+1)\# T(2,2n+1)$, with $m,n\geq 0$,
\item $P(2p, 2q+1, 2r+1, 1)$, with $p\geq 1$, $q,r\geq 0$,
\item $T(3,4)$, or $T(3,5)$.
\end{itemize}
\end{theorem}
The question arises whether the equivalence of $|\sigma(K)| = 2g(K)$ and $\gtop(K)=g(K)$ holds for other families of knots $K$. Indeed, it is also true for braid positive knots $K$~\cite{L}. Moreover, we do not know if there exists a knot $K$ of braid index $4$ with $\gtop(K) = g(K)$, but $|\sigma(K)| < 2g(K)$. One may check that such a knot would have to be prime and have crossing number at least 13.
For braid index $5$ however, there are several knots $K$ in the table with $\sigma(K) = 0 < 2g(K) = 2$ and $\gtop(K) = g(K) = 1$, such as $K = 8_1$~\cite{knotinfo}.

The proof of \cref{thm:1} is based on a technique called twisting, used by McCoy for estimating the topological 4-genus from above \cite{McC}. We will also make use of a special presentation for 3-braids called Xu normal
form, which we will describe in the next section. The third section contains the proof of \cref{thm:1}, as well as a complementary result (\cref{thm:2}): a sharp lower bound on the difference $g(K)-\gtop(K)$ of so-called strongly
quasipositive 3-braid knots, in terms of two characteristic quantities associated with their Xu normal form. 
In the last section, we determine the 
topological 4-genus of various families of braid positive 3-braid knots (almost) exactly, and display examples where our technique comes to a limit.\\

\emph{Acknowledgements:} The second author is supported by the Emmy Noether Programme of the DFG, Project number 412851057. 
The fourth author thanks Peter Feller for his constant encouragement and gratefully acknowledges support from the Swiss National Science Foundation Grant 181199.

\section{The Xu normal form of 3-braids}
\label{sec:Xu}
Our tool to handle 3-braids is what we call their \emph{Xu normal form}.
It was developed by Xu~\cite{Xu} (who called it \emph{representative symbol}),
as a variation of the Garside normal form~\cite{G}.
Using the Xu normal form, one may decide whether two given 3-braids are conjugate~\cite{Xu}, and whether their closures are equivalent links~\cite{BM1,BM2}. Later, the Xu normal form was generalized to braids on arbitrarily many strands by Birman--Ko--Lee~\cite{BKL}.
The Garside, Xu and BKL normal forms are all examples of Garside structures on (braid) %
groups~\cite{DDGKM}.
In this section, we will introduce the Xu normal form
and show how it determines the signature invariant $\sigma$ of its closure 
(\cref{prop:signature}),
as well as strong quasipositivity and braid positivity~(\cref{prop:positivity}).

A \emph{3-braid} $\beta$ is an element of the braid group $B_3 = \langle a,b \mid aba = bab \rangle$.
By closing off the three strands of $\beta$ and 
interpreting $a$ and $b$ as positive crossings between the first two and last two strands, respectively,
$\beta$ gives rise to a link called \emph{closure of $\beta$}, denoted $\cl(\beta)$.
Note that throughout the text, links are oriented.
Let us write $x \coloneqq a^{-1}ba \in B_3$ and $\delta \coloneqq ba = ax = xb \in B_3$. 
\begin{figure}[ht]
\begin{center}
\includegraphics[scale=1]{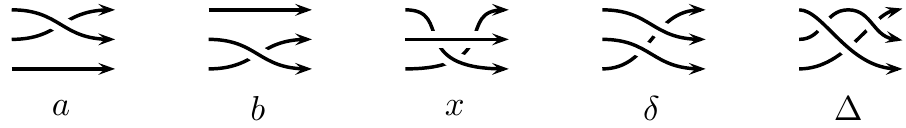}
\caption{The generators $a$ and $b$ of the braid group $B_3$ and the elements $x=a^{-1}ba$, $\delta=ba$ and $\Delta=aba$. The latter is used for the Garside normal form (further down in the text).} \label{fig:abxdelta}
\end{center}
\end{figure}
By a \emph{Xu word} or simply \emph{word}, we mean a word with letters $a,b,x,\delta$ (see \cref{fig:abxdelta}) and their inverses.
We reserve the equality sign $=$ for equality of braids, and use a dotted equality sign $\doteq$ for equality of words.
Moreover, we write $\beta \sim \gamma$ if the two braids $\beta, \gamma$ are conjugate. 
For efficiency, let us also introduce the following notation:
for any $i\in \Z$, set
$\tau_i \doteq a$ if $i\equiv 1 \pmod{3}$, $\tau_i \doteq b$ if $i\equiv 2 \pmod{3}$ and $\tau_i \doteq x$ if $i \equiv 0 \pmod{3}$.
\begin{definition}\label{def:bkl}
Let $w$ be a word of the form
\begin{equation}\label{eq:bkl}
w \doteq \delta^n \tau_1^{u_1} \tau_2^{u_2} \ldots \tau_t^{u_t}
\quad\text{for}\quad n\in \Z, t \geq 0, u_i \geq 1.
\end{equation}
We say that $w$ is in \emph{Xu normal form}
if the tuple $(-n,t,u_1, \ldots, u_t)$ is lexicographically minimal among all
words of the form \eqref{eq:bkl} representing the same conjugacy class of 3-braids.
\end{definition}
The condition of lexicographic minimality means, in particular, that
the Xu normal form maximizes $n$, and afterwards minimizes $t$.
The term `normal form' is justified by the following.
\begin{theorem}\label{thm:xuform}
Every 3-braid is conjugate to a unique word in Xu normal form.
\qed
\end{theorem}
The following lemma gives a criterion to easily decide whether a word is in Xu normal form.
\begin{lemma}\label{lem:xuform}\label{lem:bkl}
A word $w \doteq \delta^n \tau_1^{u_1} \tau_2^{u_2} \ldots \tau_t^{u_t}$
for some $n\in \Z$, $t \geq 0$, $u_i \geq 1$
is in Xu normal form if and only if one of the following conditions holds:
\begin{enumerate}[label=(\alph*)]
\item $t = 0$. In this case $w \doteq \delta^n$.
\item $t = 1$, and if $n\equiv 1\pmod{3}$ then $u_1 = 1$.
In this case, $w \doteq \delta^{3k} a^{u_1}$, or $w \doteq \delta^{3k+1} a$, or $w \doteq \delta^{3k+2} a^{u_1}$.
\item $t \geq 2$, $n + t \equiv 0\pmod{3}$ and the tuple $(u_1, \ldots, u_t)$ is lexicographically minimal among its cyclic permutations.
\end{enumerate}
\end{lemma}

The proofs of \cref{thm:xuform} and \cref{lem:xuform} are essentially contained in Xu's paper \cite[Section~4]{Xu}, albeit with slightly different conventions.
In our setup, \cref{thm:xuform} is not actually hard to prove, and makes a good exercise to get acquainted with the calculus of Xu words. The same is true for the `only if' direction of \cref{lem:bkl}.
Let us provide two hints. Firstly, the easily verifiable rules
\begin{equation}\label{eq:xurules}
\delta = \tau_i \tau_{i-1},\qquad
\tau_i\delta^n = \delta^n\tau_{i+n},\qquad
\tau_i^{-1}\delta^n = \delta^{n-1}\tau_{i+n+1}
\end{equation}
allow
to find Xu words for every 3-braid without the letters $a^{-1}, b^{-1}, x^{-1}$,
and to `pull all $\delta^{\pm 1}$ to the left' in a Xu word.
In this way, one can find a Xu word of the form 
$\delta^n \tau_{m+1}^{u_{1}} \ldots \tau_{m+t}^{u_{t}}$
with $m\in\Z$ and $u_i \geq 1$ for any 3-braid. Secondly, note that
\begin{multline*}
\delta^n \tau_1^{u_1} \ldots \tau_t^{u_t} \ =\ 
\tau_{1-n}^{u_1}\delta^n \tau_2^{u_2} \ldots \tau_t^{u_t}
\ \sim\  \delta^n \tau_2^{u_2} \ldots \tau_t^{u_t} \tau_{1-n}^{u_1}\\
\sim\ \delta^{n+1} \tau_2^{u_2} \ldots \tau_t^{u_t} \tau_{1-n}^{u_1} \delta^{-1}
\ =\ \delta^n \tau_1^{u_2} \ldots \tau_{t-1}^{u_t} \tau_{-n}^{u_1}.
\end{multline*}
So cyclically permuting the tuple $(u_1, \ldots, u_t)$
results in a
conjugate braid if $-n \equiv t\pmod{3}$.
\begin{proof}[Proof of the `if' direction of \cref{lem:bkl}]
Xu proves that condition (c) is sufficient for $w$ to be in Xu normal form (see the definition of the representative symbol and Theorem~5 in ~\cite{Xu}), but omits a discussion of conditions (a) and (b).
So suppose $w$ satisfies (a) or (b), $w' = \delta^{n'} \tau_1^{u'_1} \ldots \tau_{t'}^{u'_{t'}}$
is in Xu normal form, and $w'\sim w$. We need to show that $w \doteq w'$.
Let us distinguish the following cases.
\begin{itemize}[leftmargin=0pt,itemindent=3.7em]
\item If $w$ satisfies (a), then $w \doteq \delta^n$. 
The words $w$ and $w'$ must have the same {\em writhe}, that is, the same image under the abelianisation homomorphism $B_3\to\Z$, which maps both $a$ and $b$ to $1$. Therefore, $2n = 2n' + u'_1 + \cdots + u'_{t'}$.
But since $w'$ is in Xu normal form, $n' \geq n$. Because $u'_i \geq 0$, it follows that $n' = n$ and $t' = 0$, thus $w \doteq w'$ as desired.
\item If $w$ satisfies (b) and $w\doteq \delta^n a$, a similar argument applies:
the equality of the
writhes of $w$ and $w'$ now reads as 
$2n + 1 = 2n' + u'_1 + \cdots + u'_{t'}$. Since the left hand side is odd, we must have $t'\geq 1$ and $u'_1\geq 1$, so $2n+1\geq 2n'+1$. On the other hand $n'\geq n$, again because $w'$ is in Xu normal form, so $2n+1\leq 2n'+1$. It follows that $n'=n$, $t' = 1, u'_1 = 1$ and $w \doteq w'$ as desired.
\item The remaining case is that $w \doteq \delta^n a^{u_1}$ satisfies (b), $u_1 \geq 2$, and $n\not\equiv 1\pmod{3}$.
Now the so-called $r$-index of $w$ defined below Lemma~5 in~\cite{Xu} is~0,
and so \cite[Theorem~4]{Xu} implies that $n$ is maximal. It follows that $n' = n$,
and hence the minimality of $t'$ implies $t' \leq t = 1$. 
Finally, from the equality of the writhes of $w$ and $w'$, it follows that $t' = 1$ and $u_1 = u'_1$, and thus $w \doteq w'$.\qedhere
\end{itemize}
\end{proof}
To get a link invariant from the Xu 
normal form,
we need to understand the relationship between conjugacy classes of 3-braids and link equivalence classes of their closures. Birman--Menasco have shown that with a few well-understood exceptions, this relationship is one-to-one:
\begin{figure}[b]
\begin{center}
\includegraphics[scale=1.1]{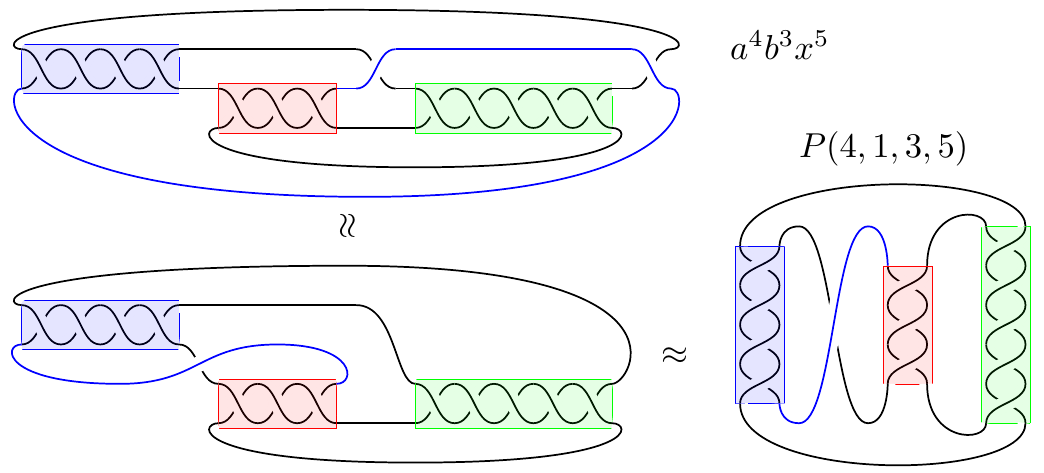}
\caption{Isotopy (denoted $\approx$) from the closure of the braid $a^{u_1}b^{u_2}x^{u_3}$ to the pretzel knot $P(u_1,1,u_2,u_3)\approx P(u_1,u_2,u_3,1)$; here $(u_1,u_2,u_3)=(4,3,5)$.} \label{fig:pretzel}
\end{center}
\end{figure}
\begin{theorem}[\cite{BM1,BM2}]\label{thm:braidtolink}
Two 3-braids are conjugate if their closures are equivalent links, except in the following cases:
\begin{enumerate}
\item The non-conjugate braids $ab, ab^{-1}, a^{-1} b^{-1}$ have the unknot as closure.
\item For $n\in\Z\setminus\{\pm 1\}$, the non-conjugate braids $a^n b, a^n b^{-1}$ have the $T(2,n)$ torus link as closure.
\item For pairwise distinct integers $p,q,r\in \Z\setminus \{0,-1,-2\}$, the two non-conjugate braids $\beta = a^p b^q x^r$ and $\gamma = a^p b^r x^q$ have the $P(p,q,r,1)$ pretzel link as closure (see \cref{fig:pretzel}); and the two non-conjugate braids $\beta^{-1}$ and $\gamma^{-1}$ have the $P(-p,-q,-r,-1)$ pretzel link as closure.
\end{enumerate}
\end{theorem}
The following corollary allows us to sidestep the exceptional cases (1), (2), (3) in the above theorem
by focusing on links of braid index 3 instead of 3-braid links
(the latter class of links includes links with braid index 1 and 2, i.e.~2-stranded torus links).
Let the \emph{reverse} of a braid $\beta\in B_3$, denoted by $\rev(\beta)$,
be the braid given by reading $\beta$ backwards and switching $a$ with $b$,
and $a^{-1}$ with $b^{-1}$. Note that $\cl(\rev(\beta))$ is obtained from $\cl(\beta)$ by reversing the link's orientation.
\begin{corollary}\label{cor:xu}
Let $L$ be a link of braid index 3.
\begin{enumerate}
\item Either there is a unique conjugacy class of 3-braids with closure $L$,
or there are two of them, such that one consists of the reverses of the braids contained in the other.
\item
The numbers $n$, $t$, and $U = u_1 + \cdots + u_t$
of the Xu normal form of a braid with closure $L$
do not depend on the choice of braid.
Thus $n$, $t$ and $U$ are link invariants of links with braid index 3.
\end{enumerate}
\end{corollary}
\begin{proof}
Claim (1) follows quickly from \cref{thm:braidtolink},
since the unknot and the two-stranded torus links have braid index less than 3,
and $\rev(a^p b^q x^r) = x^r a^q b^p \sim a^p b^r x^q$.
For (2), note that $\rev(x) = x$ and so $\rev(\tau_i) = \tau_{-i}$. Also $\rev(\delta) = \delta$.
Thus the reverse of $\delta^n \tau_1^{u_1} \tau_2^{u_2} \ldots \tau_t^{u_t}$ is
$\tau_{-t}^{u_t} \ldots \tau_{-1}^{u_1} \delta^n$, which has Xu normal form
$\delta^n \tau_1^{u_t} \tau_2^{u_{t-1}} \ldots \tau_t^{u_1}$ (up to cyclically permuting the exponents $u_t, \ldots, u_1$).
So the numbers $n$, $t$ and $U$ do not change under braid reversal. Together with (1), this implies (2).
\end{proof}

Xu calls $n$ and $t$ the \emph{power} and \emph{syllable length},
while Birman--Ko--Lee use the terms \emph{infimum} and \emph{canonical length}, respectively.

Since the Xu normal form determines the link type, all link invariants may be read off it.
Let us first prove a formula for the signature invariant.
We will need the Garside normal form \cite{G}, which we introduce now. The reader will note many parallels between the Garside and Xu normal forms.
Let $\Delta\coloneqq aba$, as in \cref{fig:abxdelta} on the right. A \emph{Garside word} is a word with letters $a, b, \Delta$ and their inverses. Again, we use $\doteq$ for equality of words.
We also use the following notation:
for any $i\in \Z$, set $\sigma_i \doteq a$ if $i\equiv 1 \pmod{2}$ and $\sigma_i \doteq b$ if $i\equiv 0 \pmod{2}$.
\begin{proposition}[{\cite[Proposition 3.2]{T}}]\label{prop:garside}
Every 3-braid contains in its conjugacy class a unique Garside word $v$ in
\emph{Garside normal form}, i.e.\ a word
\[
v \doteq \Delta^{\ell} \sigma_1^{p_1} \sigma_2^{p_2} \ldots \sigma_r^{p_r},
\]
with $\ell \in \Z$, $r\geq 0$, $p_i\geq 1$, satisfying one of the following conditions:
\begin{enumerate}[leftmargin=4em]
\item[(A)] $\ell$ is even and $r\in \{0,1\}$, i.e.~$v \doteq \Delta^{2k} a^{\geq 0}$,
\item[(B)] $\ell$ is even, $r = 2$, $p_1\in \{1,2,3\}$ and $p_2 = 1$, i.e.~$v \doteq \Delta^{2k} a^{\{1,2,3\}} b$,
\item[(C)/(D)] $r \geq 1$, $p_i \geq 2$, $\ell \equiv r \pmod{2}$, and the tuple
$(p_1, \ldots, p_r)$ is lexicographically minimal among its cyclic permutations.
\end{enumerate}
We refer by (C) and (D) to the case that $\ell$ is even and odd, respectively.
\end{proposition}
The following lemma tells us how to convert between the Xu and the Garside normal forms.
\begin{lemma}\label{lem:convert}
Let a word $w \doteq \delta^n \tau_1^{u_1} \tau_2^{u_2} \ldots \tau_t^{u_t}$ in Xu normal form be given.
Then the unique word $v$ in Garside normal form representing the same conjugacy class of 3-braids as $w$
is given by the following table.\medskip

\noindent
\begin{tabular}{ll|ll}\hline\rule{0pt}{5ex}%
\parbox[b]{5em}{\raggedright Case in\\\cref{lem:bkl}} &
\parbox[b]{5em}{\raggedright Xu normal form $w$} &
\parbox[b]{7em}{\raggedright Garside normal\\ form $v$} &
\parbox[b]{7em}{Case in\\\cref{prop:garside}} \\\hline
\rule{0pt}{3ex}%
\emph{(a)} & $\delta^{3k}$           & $\Delta^{2k}$             & \emph{(A)} \\[1ex]
\emph{(a)} & $\delta^{3k+1}$         & $\Delta^{2k}ab$           & \emph{(B)} \\[1ex]
\emph{(a)} & $\delta^{3k+2}$         & $\Delta^{2k} a^3b$        & \emph{(B)} \\[1ex]
\emph{(b)} & $\delta^{3k} a^{u_1}$   & $\Delta^{2k} a^{u_1}$     & \emph{(A)} \\[1ex]
\emph{(b)} & $\delta^{3k+1} a$       & $\Delta^{2k} a^2 b$       & \emph{(B)} \\[1ex]
\emph{(b)} & $\delta^{3k+2} a^{u_1}$ & $\Delta^{2k+1} a^{1+u_1}$ & \emph{(D)} \\[1ex]
\emph{(c)} & $\delta^n \tau_1^{u_1} \ldots \tau_t^{u_t}$
                              & $\Delta^{(2n-t)/3} \sigma_1^{1+u_1} \ldots \sigma_t^{1+u_t}$
                                                          & \emph{(C)/(D)}
\raisebox{-1.5ex}{}\\\hline
\end{tabular}
\end{lemma}
\begin{proof}
All rows in the table except for the last one may be checked quickly, using $\delta^3 = \Delta^2$.
Let us now prove $w\sim v$ for the words $w$ and $v$ in the last row.
Let $n = 3k + m$ and $t = 3s + 3 - m$ for $k,s\in\Z, s\geq 0, m\in\{1,2,3\}$.
In the Xu word $w$, replace $\delta^n$ by $(ba)^m\Delta^{2k}$.
Moreover, replace every $x^u$ by $\Delta^{-1} ab^{1+u} a$.
These replacements yield a Garside word $v_1$ with $v_1 = w$ and
\[
v_1 \doteq (ba)^m \Delta^{2k} a^{u_1} b^{u_2} (\Delta^{-1} ab^{1+u_3}a) a^{u_4} \ldots
(\Delta^{-1} ab^{1+u_{3s}}a) a^{u_{3s+1}}b^{u_{3s+2}},
\]
where we set $u_i = 0$ if $i > t$.
Now proceed by `pulling all the $\Delta^{-1}$ to the right', i.e.~replacing
$\Delta^{-1}a$ by $b\Delta^{-1}$ and $\Delta^{-1}b$ by $a\Delta^{-1}$ as long as possible.
These replacements produce a word $v_2$ with $v_2 = v_1$, where $v_2$ starts with $(ba)^m\Delta^{2k} a^{u_1} b^{u_2} (ba^{1+u_3}b) b^{u_4} \ldots$.
Using the $\sigma_i$-notation and noting that there are precisely $s$ occurrences of $\Delta^{-1}$ in $v_1$, we have
\begin{align*}
v_2 & \doteq (ba)^m\Delta^{2k} \sigma_1^{u_1} \sigma_2^{1+u_2} \sigma_3^{1+u_3} \sigma_4^{1+u_4} \ldots \sigma_{3s}^{1+u_{3s}} \sigma_{3s+1}^{1+u_{3s+1}}\sigma_{3s+2}^{u_{3s+2}} \Delta^{-s} \\
\sim v_3  & \doteq \Delta^{-s}(ba)^m\Delta^{2k} \sigma_1^{u_1} \sigma_2^{1+u_2}  \ldots \sigma_{3s}^{1+u_{3s}} \sigma_{3s+1}^{1+u_{3s+1}}\sigma_{3s+2}^{u_{3s+2}}.
\end{align*}
Let us now consider the three possibilities for $m$ case by case.
\begin{itemize}[leftmargin=0pt,itemindent=3.7em]
\item If $m = 3$, then $u_{3s+2} = u_{3s+1} = 0$ and
\begin{align*}
v_3 & \doteq \Delta^{-s}(ba)^3\Delta^{2k} \sigma_1^{u_1} \sigma_2^{1+u_2} \ldots \sigma_{3s}^{1+u_{3s}} \sigma_{3s+1} \\
    & = \Delta^{2k + 2 - s} \sigma_1^{u_1} \sigma_2^{u_2 + 1} \ldots \sigma_{3s}^{1+u_{3s}}\sigma_{3s+1} \\
\sim v_4  & = \Delta^{2k + 2 - s} \sigma_1^{1+u_1} \sigma_2^{u_2 + 1} \ldots \sigma_{3s}^{1+u_{3s}}.
\end{align*}
We have $v_4 \doteq v$ as desired, since $2k + (m - 1) - s = (2n-t)/3$.
\item If $m = 2$, then $u_{3s+2} = 0$ 
and
\begin{align*}
v_3 & \doteq \Delta^{-s}(ba)^2\Delta^{2k} \sigma_1^{u_1} \sigma_2^{1+u_2} \ldots \sigma_{3s}^{1+u_{3s}} \sigma_{3s+1}^{1+u_{3s+1}} \\
    & = \Delta^{2k + 1 - s} \sigma_1^{1+u_1} \sigma_2^{1+u_2} \ldots \sigma_{3s}^{1+u_{3s}} \sigma_{3s+1}^{1+u_{3s+1}} \ \doteq v.
\end{align*}
\item If $m = 1$, then
\begin{align*}
 v_3  & \doteq \Delta^{-s}ba\Delta^{2k} \sigma_1^{u_1} \sigma_2^{1+u_2}  \ldots \sigma_{3s}^{1+u_{3s}} \sigma_{3s+1}^{1+u_{3s+1}} \sigma_{3s+2}^{u_{3s+2}} \\
      & = \sigma_{s} \Delta^{2k - s} \sigma_1^{1+u_1}\sigma_2^{1+u_2}  \ldots  \sigma_{3s+1}^{1+u_{3s+1}} \sigma_{3s+2}^{u_{3s+2}} \ \sim v.
\qedhere
\end{align*}
\end{itemize}
\end{proof}
We rely on the signature formula for 3-braids in Garside normal form deduced by the fourth author
\cite[Remark~1.6, Proposition~4.2, Remark~4.3]{T} from a result by Erle~\cite[Theorem~2.6]{E}.%
\begin{proposition}[\cite{T}]\label{prop:garsidesig}
Let $K$ be a knot that is the closure of a Garside normal form $\Delta^{\ell} \sigma_1^{p_1} \ldots \sigma_{r}^{p_r}$ in case (C)/(D) of \cref{prop:garside}. Then
\[
\sigma(K) = -2\ell + r - \sum_{i=1}^r p_i.
\]
\end{proposition}

We are now ready to state and prove our signature formula for 3-braids in Xu normal form.
\begin{proposition}\label{prop:signature}
Let $\delta^n \tau_1^{u_1} \tau_2^{u_2} \ldots \tau_t^{u_t}$ be the Xu normal form
of a 3-braid whose closure is a knot $K$ of braid index 3.
Set $U = u_1 + \cdots + u_t$. If $t > 0$ (equivalently, if $K$ is not a torus knot), then
\[
\sigma(K) = -U - \frac43 n + \frac23t.
\]
In the case $t = 0$, i.e.\ $U = 0$ and $K = T(3,n)$, the value $\sigma(K) =  -\frac43 n$
given by the above formula is only approximately true, with an error of at most $\frac43$.
In fact, 
in that case we have
\[
\sigma(K) = - \frac43 n + \left(2 + 4\left\lfloor\frac16 n\right\rfloor - \frac23 n\right)
=
2 - 2n + 4\left\lfloor\frac16 n\right\rfloor.
\]
\end{proposition}
\begin{proof}
Our signature formula for torus knots may be seen to agree with
the formula given e.g.~in~\cite[Proposition~9.1]{M2}.
So we are left with the case $t\geq 1$, i.e.~cases (b) and (c) in \cref{lem:bkl}.
Denote the Xu normal form in question by $w$.
Let us first consider case~(c).
Then $w$ has Garside normal form $\Delta^{(2n-t)/3} \sigma_1^{1+u_1}\ldots\sigma_t^{1+u_t}$,
see \cref{lem:convert}. By \cref{prop:garsidesig},
we have $\sigma(K) = -4n/3 + 2t/3 + t - (U + t)$, which is equal to the claimed formula.
In case~(b), since the closure of $w$ is a knot, the only possibility is $w \doteq \delta^{3k+2} a^{u_1}$. Then the Garside normal form of $w$ is $\Delta^{2k+1} a^{1+u_1}$,
which has the desired signature, again by \cref{prop:garsidesig}.
\end{proof}
Next, let us give complete criteria to decide braid positivity and strong quasipositivity for links of braid index 3.
A \emph{braid positive} link is the closure of some \emph{positive word}, i.e.~a word in positive powers of the standard generators $a_1, \ldots, a_{k-1}$ of the braid group $B_k$ on some number $k$ of strands. Similarly, a link is called \emph{strongly quasipositive} \cite{R1,R2} if it is the closure of a \emph{strongly quasipositive word} in some $B_k$, i.e.~a word in positive powers of
\[
a_{ij} = a_i^{-1} a_{i+1}^{-1} \ldots a_{j-1}^{-1} a_j a_{j-1} \ldots a_i
\]
with $1\leq i \leq j \leq k - 1$.
Note that for $k = 3$, positive words are words in $a = a_1$, $b = a_2$ and strongly quasipositive words are words in $a = a_{11}$, $b = a_{22}$ and $x = a_{12}$.
It is well-known and straightforward to show that a $k$-braid can be written as a (strongly quasi-)positive braid word if and only if the power of $\Delta$ ($\delta$) in its Garside (Birman--Ko--Lee)
normal form is non-negative, respectively.
This makes (strong quasi-)positivity decidable for braids.
For links however, the problem is harder because
a priori, a braid positive (strongly quasipositive) link with braid index $k$
need not be the closure of a (strongly quasi-)positive word on $k$ strands.
For $k = 3$, however, this is the case:
\begin{theorem}[{\cite[Theorem~1.1 and 1.3]{S}}]\label{thm:posstoimenow}
The following hold. 
\begin{enumerate}
\item If a strongly quasipositive link is the closure of some 3-braid, then it is the closure of a strongly quasipositive 3-braid.
\item If a braid positive link is the closure of some 3-braid, then it is the closure of a positive 3-braid.
\end{enumerate}
\end{theorem}
We are now ready to prove our positivity characterizations. For braid positivity, we will once again need the Garside normal form introduced above.%
\begin{proposition}\label{prop:positivity}
Let $\delta^n \tau_1^{u_1} \tau_2^{u_2} \ldots \tau_t^{u_t}$ be in Xu normal form,
with closure a link $L$ of braid index 3. Then the following hold.
\begin{enumerate}
\item $L$ is a strongly quasipositive link if and only if $n \geq 0$.
\item $L$ is a braid positive link if and only if $n \geq t/2$ or $n = 0, t = 1$.
\end{enumerate}
\end{proposition}
\begin{proof}
{\em Part (1).} 
If $n\geq 0$, then the Xu normal form yields a strongly quasipositive word for $L$, so $L$ is strongly quasipositive. For the other direction, assume $L$ is strongly quasipositive and a Xu normal form $w$ with closure $L$ is given. By \cref{thm:posstoimenow}, there is a strongly quasipositive word in $B_3$ representing~$L$. It may be transformed to its Xu normal form $w'$ just by replacing $\tau_{i+1}\tau_i \to \delta$ and $\tau_i \delta \leftrightarrow \delta \tau_{i+1}$, and by passing from $y\tau_i$ to $\tau_i y$ for $y$ a word in positive powers of $a,b,x,\delta$. None of these transformations create negative powers of~$\delta$, and so we find that the Xu normal form $w'$ has $n\geq 0$.
By \cref{cor:xu}(2), all Xu normal forms with closure $L$ have the same $n$, so $w$ has $n\geq 0$ as well, which was to be proven.

{\em Part (2).} If $n \geq t/2$, then one may check using \cref{lem:convert} that the Garside normal form 
of $\delta^n \tau_1^{u_1} \tau_2^{u_2} \ldots \tau_t^{u_t}$ starts with a non-negative power of~$\Delta$,
and thus yields a positive word with closure~$L$.
If $n = 0$ and $t = 1$, then the Xu normal form is $a^{u_1}$, which is already a positive word with closure~$L$. In both cases, it follows that $L$ is braid positive.
For the other direction, assume $L$ is braid positive and let a Xu normal form $w$ with closure $L$ be given. By \cref{thm:posstoimenow},
there is a positive word in $B_3$ representing~$L$.
Similarly as in the proof of part (1), one sees that the Garside normal form of this positive word starts with $\Delta^{\ell}$ with $\ell \geq 0$. By going through the rows of the table in \cref{lem:convert}, one sees that this implies that $n \geq t/2$, with the sole exception in the fourth row if $k = 0$ and $u_1\neq 0$: then, the Xu normal form is $a^{u_1}$, so $n = 0$ and $t = 1$.
Again using \cref{cor:xu}(2), it follows that $w$ also satisfies $n \geq t/2$, or $n = 0$ and $t = 1$.
\end{proof}
If a knot $K$ is the closure of a strongly quasipositive 3-braid in Xu normal form $\delta^n \tau_1^{u_1}\ldots \tau_t^{u_t}$, there is also a simple formula for the Seifert genus of $K$:
\begin{equation}\label{eq:genus}
g(K) = \frac{U}{2} + n - 1,
\end{equation}
where 
$U = u_1 + \cdots + u_t$. This follows from the Bennequin inequality~\cite{bennequin}. 

\section{Proofs of main theorems} 
\label{sec:proofs}
Before beginning with the proof of \cref{thm:1}, we describe a technique that we use to detect topological 4-genus defect in a given knot $K$, that is, to show $\gtop(K)<g(K)$. The main ingredient is the so-called {\em generalized crossing change}, also known as {\em null-homologous twist}, or simply {\em twist}. A null-homologous twist consists in performing a $\pm 1$ Dehn surgery on the boundary circle of an embedded disc $D\subset S^3$, such that $D$ intersects $K$ transversely in a finite number of interior points, with total algebraic intersection count~$0$. While $\pm 1$ Dehn surgery on an unknot in~$S^3$ gives back~$S^3$, the effect on~$K$ is that a (left- or right-handed) full twist is introduced into the strands of~$K$ that cross~$D$ (cf.~\cref{fig:twist}).
\begin{figure}[b]
\begin{center}
\includegraphics[scale=1.3]{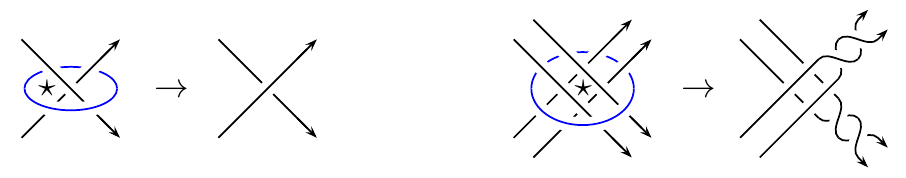}
\caption{Two examples of twists, at the locations marked by~$\star$.
Here, the boundary of the respective disc $D$ is drawn in blue; in subsequent figures, it will be omitted.}
\label{fig:twist}
\end{center}
\end{figure}
The \emph{untwisting number} $\tu(K)$ of $K$, introduced by Ince~\cite{I}, is defined as the minimal number of null-homologous twists needed to turn $K$ into the unknot. Clearly, $\tu(K)\leq u(K)$, since crossing changes are special cases of null-homologous twists. McCoy~\cite{McC} showed 
that $\gtop(K)\leq \tu(K)$, that is, the untwisting number of $K$ is an upper bound on the topological 4-genus of $K$. His result is based on Freedman's theorem, which states 
that knots of Alexander polynomial $1$ are topologically slice~\cite{F,FQ}. This bound can now be used to find topological 4-genus defect: If we find a way to turn a knot $K$ into the unknot with strictly less than $g(K)$ null-homologous twists, this will show $\gtop(K)<g(K)$. This method was already applied by Baader--Banfield--Lewark~\cite{BBL} to $3$-stranded torus knots. For the proof of \cref{thm:1} below, we use a slightly refined version of the method, as follows.

\begin{figure}[t]
\begin{center}
\includegraphics[scale=1.5]{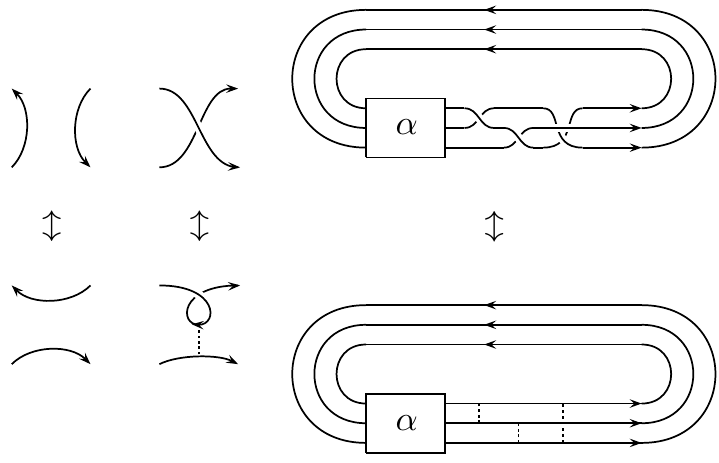}
\caption{Left: Saddle move. Middle: How to use isotopy and a saddle move to add or remove a braid crossing. Right: Example of a cobordism between $\cl(\alpha)$ and $\cl(\alpha\cdot abx)$ that consists of three saddle moves, for some $3$-braid $\alpha$.} \label{fig:saddle}
\end{center}
\end{figure}

Assume we find a cobordism $C\subset S^3\x[0,1]$ from a given knot~$K$ to some knot~$K'$
such that $g(C)=g(K)-g(K')$. If such a $C$ exists, we will write~$K\leadsto K'$.
Assume that furthermore $\tu(K')<g(K')$. Then, by McCoy's result, $\gtop(K')\leq \tu(K')<g(K')$. Composing the cobordism $C$ with a topological slice surface for $K'$, we obtain
\[
\gtop(K)\leq \gtop(K')+g(C)<g(K')+g(C)=g(K).
\]
In particular, for the topological 4-genus defect we find
\begin{equation}\label{eq:defect2}
g(K) - \gtop(K)\geq g(K') - \gtop(K') \geq g(K') - \tu(K').
\end{equation}

One way to construct cobordisms is to apply saddle moves to knot diagrams, as in \cref{fig:saddle}. Note that such cobordisms will always be smooth.
Suppose the knot $K$ is the closure of a strongly 
quasipositive 3-braid of the form $\beta=\delta^n \tau_1^{u_1}\cdots\tau_t^{u_t}$ with $n\geq 0$ and $u_1, \ldots, u_t \geq 1$.
Out of saddle moves, one may build a cobordism $C$
that
lowers the exponents $u_i$ and $n$ 
(one saddle move 
per letter $\tau_i$, two saddle moves 
per letter $\delta$), or transforms $\delta$ into $\tau_i$ (one saddle move).
Suppose the exponents remain non-negative, and $C$ is a cobordism from $K$ to another knot $K'$.
Then $K'$ is also strongly quasipositive,
and it follows from the Bennequin inequality, see \eqref{eq:genus},
that $g(C) = g(K) - g(K')$, i.e.~$K \leadsto K'$.

\begin{proof}[Proof of \cref{thm:1}]
We organize the proof into two parts, which consist in verifying the following statements.

\begin{enumerate}
\item $|\sigma(K)|=2g(K)$ for all $K$ in the list of \cref{thm:1} and their mirrors,
\item $\gtop(K)<g(K)$ for all other $3$-braid knots except the figure-eight.
\end{enumerate}

In light of Kauffman and Taylor's signature bound $|\sigma(J)|\leq 2\gtop(J)$, valid for all knots $J$, see~\cite{KT,P}, these two statements together imply the theorem.\\[-0.7em]

{\em Part (1).} The genera and signatures of torus knots are well understood~\cite{Lit}; in particular, we know that the torus knots $K$ with $|\sigma(K)|=2g(K)$ are precisely the knots $T(3,4), T(3,5)$, $T(2,2n+1)$, $n\geq 0$, and their mirrors.
In fact, the signature of $T(2,2n+1)$, $n\geq 0$, is known to be $-2n$. Both Seifert genus and signature are additive under connected sum of knots; hence $|\sigma(K)|=2g(K)$ for all knots $K$ of the first type listed. If $K$ is one of the listed pretzel knots $K=P(2p,2q+1,2r+1,1)$ with $p\geq 1, q,r\geq 0$, then $K$ has a positive and alternating, hence special alternating diagram; see \cref{fig:pretzel}. Murasugi shows in this case that $|\sigma(K)|=2g(K)$; see~\cite[Corollary~10.3]{M1}.\\[-0.7em]

{\em Part (2).} 
Let $K$ be a $3$-braid knot other than the figure-eight knot such that neither $K$ nor its mirror appears in the list of \cref{thm:1}. The goal is to show that $\gtop(K)<g(K)$. We distinguish several cases.
\begin{itemize}[leftmargin=0pt,itemindent=3.7em]
\item If $K$ is the closure of a positive $3$-braid, this is a special case of the analogous statement about all positive braids, on any number of strands, by Liechti \cite[Theorem~1, Corollary~2]{L}.

\item Next, we consider the case in which $K$ is strongly quasipositive without being braid positive; this is the main part of the proof. By \Cref{lem:xuform} and \cref{prop:positivity},
$K$ is the closure of a $3$-braid $\beta$ in Xu normal form $\beta=\delta^n \tau_1^{u_1}\tau_2^{u_2}\cdots \tau_t^{u_t}$ with $u_1,\ldots,u_t\geq 1$, $n\geq 0$, $t\geq 2$, $n+t\equiv 0\mod 3$, and $2n<t$ (note that the case $n = 0, t = 1$ is excluded since we assume that $\cl(\beta)$ is a knot). These conditions on $n$ and $t$ leave the following possibilities: $(n,t)=(0,3)$, or $(n,t)=(1,5)$, or $t\geq 6$. First, if $(n,t)=(0,3)$, then $K=P(u_1,u_2,u_3,1)$; see \cref{fig:pretzel}. If more than one of $u_1,u_2,u_3$ is even, $K$ is a link of more than one component; the same is true if all three parameters are odd. We may therefore assume that $u_1=2p$ is even and $u_2=2q+1$, $u_3=2r+1$ are odd, with $p\geq 1, q,r\geq 0$, in which case $K$ is a pretzel knot from the list, which we excluded.
\begin{figure}[b]
\begin{center}
\includegraphics[scale=1.2]{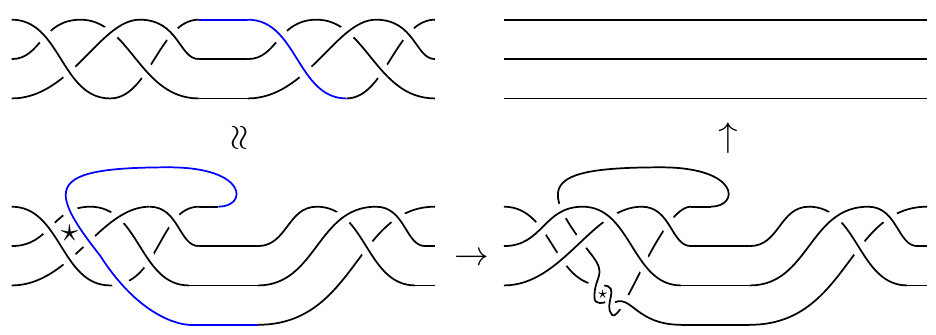}
\caption{$abxabx\twist \varnothing$ using one twist on four strands, followed by another twist on two strands, at the locations marked~$\star$ (cf.~\cref{fig:twist}). The first step is an isotopy, moving the blue strand. The last step is a crossing change near~$\star$, followed by an isotopy fixing the endpoints of the braid strands.} \label{fig:abxabx}
\end{center}
\end{figure}

Secondly, if $(n,t)=(1,5)$, we have $\beta=\delta a^{u_1}b^{u_2}x^{u_3}a^{u_4}b^{u_5}$. If the exponents $u_i$ are all odd, $\beta$ closes to a two component link, a contradiction to $K$ being a knot. Therefore at least one of the $u_i$ is even. We may assume that $u_1$ is even, because the exponents $u_1,u_2,\ldots,u_t$ may be cyclically permuted without changing the braid closure, as explained in \cref{sec:Xu} after \cref{lem:xuform}. In particular, we may assume that $u_1\geq 2$. Then, since $\delta=xb$,
\[ \beta \ \sim \ a^{u_1} b^{u_2} x^{u_3} a^{u_4} b^{u_5} x b \ \leadsto \ a^2 b x a b x b \ =\ a (abx)^2 b \ \twist \ ab, \]
where `$\sim$' denotes conjugation, `$\leadsto$' denotes the existence of a cobordism whose genus equals the difference of the Seifert genera of the knots it connects and `$\twist$' is shown in \cref{fig:abxabx}. Here, the cobordism `$\leadsto$' is built from saddles decreasing the exponents~$u_i$. Since the closure of $a(abx)^2 b$ has Seifert genus $3$ by \eqref{eq:genus}, while only $2$ twists are used in `$\twist$', we obtain $\gtop(K)\leq g(K)-1<g(K)$.

Finally, if $t\geq 6$, we proceed similarly as in the previous case. Again, the parity of the exponents $u_1, u_2, \dots, u_t$ determines whether $\beta$ closes to a knot or a multi-component link. Indeed, a brief case study shows that the closure $K$ of $\beta$ is a link of two or three components whenever all the $u_i$ are odd. We may therefore assume via conjugation that $u_1$ is even. If $t > 6$ or $n > 0$, then
\[ \beta \ \leadsto \ a^{u_1} b^{u_2} x^{u_3} a^{u_4} b^{u_5} x^{u_6} b \ \leadsto \ a (abx)^2 b \ \twist \ ab. \]
If $t = 6$ and $n = 0$, then $\beta=a^{u_1}b^{u_2}x^{u_3}a^{u_4}b^{u_5}x^{u_6}$. 
In this case, in order for $K$ to be a knot,
at least two of the exponents $u_1, \dots, u_6$, say $u_i$ and $u_j$, need to be even. We may assume that $(i,j)=(1,2)$ or $(i,j)=(1,4)$. To see this, use cyclic permutation (as in the case $(n,t)=(1,5)$) and the fact that $u_1,u_3$ cannot be the only even exponents, again because $\beta$ would not close to a knot if they were. Therefore, smooth cobordisms bring us to 
$a^2 b^2 x a b x$ or to $a^2 b x a^2 b x$, whose closures have Seifert genus $3$ by \eqref{eq:genus}. In the first case,
\[ a^2 b^2 x a b x\  =\ a^2 b \delta^{-1} \delta b x a b x \ =\ a^2 b \delta^{-1} b (abx)^2 \ \twist \ a^2 b \delta^{-1} b \ =\ a^2 b a^{-1}\ \sim \ ab. \]
For the second case, \cref{fig:abxaabx} shows how to turn $a(abxa^2bx)$ into $a\gamma$ using two twists. Here, $\gamma$ is the tangle shown in the top right corner of the figure. Note that $a\gamma$ describes the unknot when closed like a braid. Since the closure of $a^2bxa^2bx$ has Seifert genus $3$ (see \eqref{eq:genus}), we obtain $\gtop(K)\leq g(K)-1<g(K)$. This concludes the case that $K$ is strongly quasipositive.

\begin{figure}[b]
\begin{center}
\includegraphics[scale=1.2]{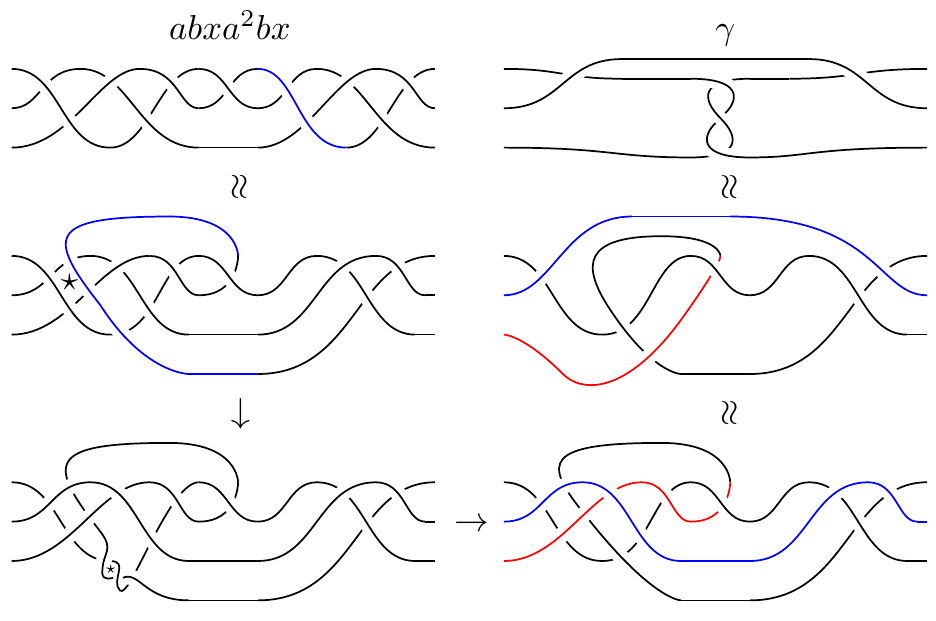}
\caption{How to turn the braid $abxa^2bx$ (top left) into the tangle $\gamma$ (top right) using one twist on four strands, followed by another twist on two strands, at the locations marked~$\star$ (cf.~\cref{fig:twist}).} \label{fig:abxaabx}
\end{center}
\end{figure}

\item If the mirror of $K$ is strongly quasipositive, we apply the above argument to the mirror of $K$; since both $\gtop$ and $g$ are invariant under taking mirror images, we  obtain $\gtop(K)<g(K)$ again.

\item If $K$ or its mirror is the knot $T(2,2m+1)\#T(2,-2n-1)$, with $m\geq n\geq 1$, it has a Seifert surface $S$ of genus $m+n=g(K)$ that contains a copy of the ribbon knot $R\coloneqq T(2,2n+1)\#T(2,-2n-1)$, bounding a subsurface of~$S$ of genus $g(R)=2n$. A surgery that cuts this subsurface off $S$ and replaces it with a slice disk for $R$ gives rise to a smooth surface of genus $m+n-2n=m-n$ embedded in the four-ball, with boundary~$K$. This shows that $g_4(K) \leq m-n \leq m+n-2 < g(K)$, because $n\geq 1$ by assumption. Here, we use the standard notation $g_4(K)$ for the smooth $4$-genus. In particular, we obtain $\gtop(K)<g(K)$.

\item What remains are the $3$-braid knots $K$ such that neither $K$ nor its mirror is among the following: strongly quasipositive, a connected sum of the form $T(2,2m+1)\#T(2,-2n-1)$, $m\geq n\geq 1$, or the figure-eight knot. For such $K$, Lee--Lee~\cite{LL} prove the following bound on the unknotting number: $u(K)<g(K)$. Since $\gtop(J)\leq g_4(J)\leq u(J)$ holds for all knots $J$, this implies $\gtop(K)<g(K)$ and completes the proof.\qedhere
\end{itemize}
\end{proof}

For comparison, we note the following analog of \cref{thm:1}, in which the smooth $4$-genus $g_4(K)$ replaces $\gtop(K)$, and Rasmussen's invariant $s(K)$
from Khovanov homology~\cite{Ras}
(or any other slice-torus invariant~\cite{Liv}, e.g.~the Heegaard--Floer $\tau$-invariant)
plays the role of the signature $\sigma(K)$.

\begin{proposition}
Let $K$ be a $3$-braid knot other than the figure-eight knot. Then
\[ |s(K)| = 2g(K) \quad \Longleftrightarrow\quad g_4(K) = g(K). \]
These equalities hold precisely when $K$ or its mirror is strongly quasipositive.
\end{proposition}

\begin{proof}
If $K$ or its mirror is strongly quasipositive, then $|s(K)|=2g(K)$ 
follows, see~\cite[Proposition~1.7]{Shu}. 
Moreover, the implication $|s(K)| = 2g(K) \Rightarrow g_4(K) = g(K)$ holds for all knots $K$,
because of the inequalities $|s(K)| \leq 2g_4(K) \leq 2g(K)$.
It remains to prove that if $K$ is a 3-braid knot other than the figure-eight knot with $g_4(K) = g(K)$, then $K$ or its mirror is strongly quasipositive. This follows from Lee--Lee's results~\cite{LL}. More precisely, for a 3-braid knot~$K$ with $g_4(K) = g(K)$ Theorem~1.1 in \cite{LL} implies $u(K)=g(K)$. By Theorem~1.3 of the same paper, $K$ or its mirror is either strongly quasipositive or a connected sum of two-strand torus knots $T(2,2m+1)\#T(2,-2n-1)$ with $m\geq n\geq 1$. The latter can be excluded as in the proof of our \cref{thm:1} by
showing $g_4(K)\leq m-n \leq m+n-2<g(K)$, a contradiction to $g_4(K)=g(K)$.
\end{proof}

\begin{theorem} \label{thm:2}
Let $K$ be a strongly quasipositive $3$-braid knot, written as the closure of a $3$-braid $\beta$ in Xu normal form $\beta=\delta^n \tau_1^{u_1}\tau_2^{u_2}\cdots \tau_t^{u_t}$, with $u_1,\ldots, u_t\geq 1$ and $n\geq 0$. Then the topological 4-genus defect of $K$ is bounded as follows: 
\[ \frac{n}{3}+\frac{t}{3}-1\geq g(K)-\gtop(K) \geq \frac{n}{3}+\frac{t}{6}-3. \]
The constants $\frac13$ and $\frac16$ in the second inequality are optimal in the following sense:
Whenever $C>\frac13$ or $D>\frac16$, and $E\in\R$, there exists a $3$-braid in Xu normal form as above with $n\geq 0$ such that its braid closure $K$ satisfies
\[ g(K)-\gtop(K) < Cn+Dt-E. \]
\end{theorem}

The case $t=0$ in \Cref{thm:2}, in which $K=T(3,n)$, is covered by \cite[Theorem~1]{BBL}. The equality
\[
g(K)-\gtop(K) = n-1 -\left\lceil \frac{2n}{3}
\right\rceil
\]
directly implies the claimed upper and lower bounds in this case.

\begin{proof}
We only consider the case $t>0$ and begin with the upper bound $\frac{n}{3}+\frac{t}{3}-1\geq g(K)-\gtop(K)$. 
We first use \cref{prop:signature} to compute the absolute value of the signature $|\sigma(K)| = -\sigma(K) = U + \frac43n - \frac23t$, where $U=u_1+\ldots+u_t$, and recall that $g(K)=\frac{U}{2} + n-1$ from \eqref{eq:genus}, \cref{sec:Xu}. The bound then follows directly from Kauffman and Taylor's classical bound $|\sigma(K)|\leq 2\gtop(K)$.

To establish the lower bound, we first apply a smooth cobordism from $K$ to a knot $K'$, by suitably lowering the exponents $n,u_1,u_2,\ldots,u_t$ in $\beta$, as explained in \cref{fig:saddle} above.
First, assume $n > 0$.
We set $K'\coloneqq \cl(\delta^m(abx)^{\frac{s}{3}})$, where
\[ s=6\left\lfloor\frac{t}{6}\right\rfloor\quad\text{and}\quad m=\left\{\begin{array}{ll}n-2 & \text{if } n\equiv 0\mod 3\\ n-1 & \text{if }n\equiv 2\mod 3\\ n & \text{if }n\equiv 1\mod 3.\end{array}\right. \]
We have $K\leadsto K'$, and so this is a smooth cobordism which does not increase the topological 4-genus defect. In other words,
\[ g(K)-\gtop(K)\geq g(K')-\gtop(K'), \]
as in \eqref{eq:defect2}.
Next, we apply the untwisting move $(abx)^2\twist\varnothing$ from \cref{fig:abxabx} exactly $\frac{s}{6}$ times to the knot $K'$, resulting in $\cl(\delta^m)=T(3,m)$. Since $m\equiv 1\mod 3$, this is a knot again. For $m\geq 4$, \cite[Lemma~5~(1)]{BBL} yields $\tu(T(3,m))\leq \frac23 m+\frac13$.
This inequality still holds for $m = 1$
and adds up to
\[ \gtop(K')\leq \tu(K')\leq 2\cdot\frac{s}{6}+\frac23m+\frac13. \]
Since $g(K')=\frac{s}{2}+m-1$ by \eqref{eq:genus}, and since $\frac{s}{6}\geq \frac{t}{6}-\frac56$ and $m\geq n-2$, we obtain
\begin{eqnarray}
g(K)-\gtop(K) 
\geq \frac{s}{6}+\frac{m}{3}-\frac43 \nonumber 
     \geq\frac{t}{6}+\frac{n}{3}-3. \nonumber
\end{eqnarray}
In the case $n = 0$, the above procedure fails because $m$, as defined above, is negative and the smooth cobordism to $K'$ might therefore increase the 4-genus defect. However, a simple cosmetic modification allows for a cobordism that only increases the 4-genus defect by at most one. Specifically, we set $m=1$ (instead of $m=-2$) and $K'=\cl(\delta(abx)^{\frac{s}{3}})$. The cobordism from $K$ to~$K'$ is then given by lowering the exponents $u_1,u_2,\ldots,u_t$ in $\beta$ as above while increasing the exponent of $\delta$ from $0$ to $1$. This gives
\[ g(K)-\gtop(K)\geq g(K')-\gtop(K')-1. \]
Now we apply the $\frac{s}{6}$ untwisting moves $(abx)^2\twist\varnothing$ as above. The result is the unknot $\cl(\delta)$. Since $n=0$, we again obtain the claimed bound
\[ g(K)-\gtop(K) \geq \frac{s}{6}-1\geq \frac{t}{6}-\frac56 -1\geq\frac{t}{6}+\frac{n}{3}-3. \]

In order to demonstrate optimality of the constants, we consider the two special families of $3$-braids $\delta^n$ and $(abx)^n$, which we slightly modify to $\delta^{3k+1}$ and $(abx)^{2k}abx^2abx^2$, to make sure that their braid closures are connected.

\begin{itemize}[leftmargin=0pt,itemindent=3.7em]
\item For $K=T(3,3k+1)$ with $k\geq 1$, the closure of the braid $\delta^{3k+1}$ 
in Xu normal form with $n=3k+1$ and $t=0$, we have $g(K)=3k$ and $\gtop(K)=2k+1$ by~\cite{BBL}, hence $g(K)-\gtop(K)=k-1$. Whenever $C>\frac13$ and $D,E$ are arbitrary constants, we will therefore have
\[ g(K)-\gtop(K)=k-1<C\cdot (3k+1)+D \cdot 0-E \]
for sufficiently large $k$.

\item If $K$ is the closure of $(abx)^{2k}abx^2abx^2$, which is in Xu normal form with $n=0$ and $t=6k+6$, then $g(K)=3k+3$.
We will now invoke the Levine--Tristram signature function \cite{Lev,Tri}, which associates to each knot $J$ a function $[0,1]\to\Z$, $t\mapsto \sigma_{e^{2\pi i t}}(J)$ with $\sigma_{1}(J)=0$ for $t=0$.
By Gambaudo--Ghys~\cite[Corollary~4.4]{GG},
for any $3$-braid~$\beta$ with closure $J$, the Levine--Tristram signature function of $J$ 
grows linearly on $(0,\frac13)$ with slope $-2$ times the writhe of $\beta$, up to a pointwise error of at most $2$ (see e.g.~\cref{fig:sigprofile} at the end of the paper). For strongly quasipositive $\beta$ with $J$ a knot, that slope is $-4(g(J)+1)$, see
\eqref{eq:genus}, and hence
\begin{equation}\label{eqn:sigmahat}
\widehat{\sigma}(J)\coloneqq \max_{\omega\in S^1\setminus\Delta_J^{-1}(0)} |\sigma_\omega(J)|\geq \frac43(g(J)+1)-2.
\end{equation}
Since $\frac12|\sigma_\omega(J)|\leq \gtop(J)$ if $\omega\in S^1$ is not a root of the Alexander polynomial of~$J$ (see~\cite{KT,P}), we obtain $\frac23(g(K)+1)-1=2k+\frac53\leq \gtop(K)$. Hence, if $C,E$ are arbitrary constants and $D>\frac16$, we have
\[ g(K)-\gtop(K)\leq 3k+3-2k-\frac53=k+\frac43 < C\cdot 0 + D\cdot (6k+6) - E, \]
for sufficiently large $k$. In this case, it does not suffice to consider the (classical) signature bound on $\gtop(K)$. Indeed, $|\sigma(K)|=2k+4$ (see~\cref{prop:signature}) is roughly half the Levine--Tristram signature $\sigma_{e^{2\pi i/3}}(K)$. Substituting $\sigma(K)$ for $\sigma_{e^{2\pi i/3}}(K)$ in the above argument would therefore not work.
\qedhere
\end{itemize}

\begin{remark}\label{rmk:abxexact}
The proof of \Cref{thm:2} shows that for $K$ the closure of $(abx)^{2k}abx^2abx^2$, $k \geq 0$, we have 
\[
2k+\frac53\leq \frac12|\sigma_{e^{2\pi i/3}}(K)| \leq \frac12\widehat{\sigma}(K)  \leq\gtop(K).
\]
On the other hand, from K 
we obtain the closure of $abx^2abx^2 \sim a^2bxa^2bx$ using $2k$ twists (using $k$ times our trick $abxabx\twist \varnothing$ from \Cref{fig:abxabx}), which can be untwisted with two twists (compare the case $t=6, n=0$ in the proof of \Cref{thm:1}). Thus $\gtop(K)\leq \tu(K) \leq 2k+2$ and we obtain
\[
\frac12|\sigma_{e^{2\pi i/3}}(K)| = \frac12\widehat{\sigma}(K)  =\gtop(K) = \tu(K) = 2k+2.
\]
\end{remark}
\end{proof}

\section{The 4-genus of closures of positive 3-braids} 
Our results so far provide infinitely many strongly quasipositive 3-braid knots $K$
satisfying $\frac12 \hatsigma(K) = \gtop(K)$, see \cref{theorem1} and \cref{rmk:abxexact}.
This leads us to the following.
\begin{question}\label{q:equality}
Does the equality $\frac12 \hatsigma(K) = \gtop(K)$ hold for all strongly quasipositive 3-braid knots?
If not, does it at least hold for all braid positive 3-braid knots?
\end{question}
In this section, we exhibit large families of positive 3-braid knots
for which $\frac12 \hatsigma(K) = \gtop(K)$ holds up to an error of $1$,
see \Cref{prop:braidposbound}.
In certain cases, we can even determine $\gtop$ exactly (see \Cref{ex:exactex1,ex:exactex2} and \Cref{rem:exactfamilies}). 
For the knots covered by \Cref{prop:braidposbound}, $\widehat{\sigma}$ equals $|\sigma|$ or $|\sigma|+2$,
whereas in the example discussed in \cref{rmk:abxexact}, $\widehat{\sigma}$ equals $|\sigma_{\zeta3}|$.
However, as we will see in \cref{rem:sign}, the differences
$\widehat{\sigma} - |\sigma|$ and
$\widehat{\sigma} - |\sigma_{\zeta3}|$ can be arbitrarily large for positive 3-braid knots.
This fact contributes to the difficulty of \cref{q:equality}.

Throughout, we will use the Xu normal form of $3$-braids from \Cref{sec:Xu}. 
Recall that we write $\delta = ba = ax = xb$ such that $a\delta = \delta b$, $b \delta = \delta x$ and $x \delta = \delta a$ (see \eqref{eq:xurules} in \Cref{sec:Xu}).

\begin{example}\label{ex:exactex1}
Let $K$ be a knot that is the closure of a $3$-braid in Xu normal form $\beta = \delta^{3\ell +2} a^{u_1}$ for $\ell \geq 0, u_1 \geq 1$. Note that $u_1$ must be even for $K$ to be a knot. We claim that
\begin{align}\label{eq:g4topex1}
\gtop(K) = \tu(K) =\frac{|\sigma(K)|}{2}=  \frac{u_1}{2} +2\ell +1.
\end{align}
The last equality follows directly from \Cref{prop:signature}.
To prove \Cref{eq:g4topex1}, using $\frac12 |\sigma(K)| \leq \gtop(K) \leq \tu(K)$ (as explained in the beginning of \cref{sec:proofs}) 
it is enough to 
show that $\tu(K) \leq \frac12 |\sigma(K)|$. 
By $\frac{u_1-2}{2}$ crossing changes from positive crossings of $\beta$ to negative crossings we obtain the braid $\delta^{3\ell +2} a^{2}$.
We will prove by induction that this braid can be untwisted with $2\ell +2$ twists, which implies $\tu(K) \leq \frac{u_1}{2}+2\ell+1$. For $\ell = 0$, we have $\delta^2 a^2 = baba^3 = aba^4$ which becomes $ab$ (with closure the unknot) by two crossing changes. For $\ell = 1$, we have 
\begin{align*}
\delta^5 a^2 &= \delta^4 ax a^2 
= \delta^3 x^2 bxa^2 
= \delta^2 b^3 abxa^2 
= \delta a^4 xabxa^2 
= x^5 bxabxa^2\\
&\sim b^5 abxabx x \twist b^5 x,
\end{align*}
which turns into $bx\sim \delta$ using two crossing changes. Recall that the two twists needed to untwist $abxabx \twist \varnothing$ are shown in \cref{fig:abxabx} of \Cref{sec:proofs}.
Now, for $\ell \geq 2$, we have
\begin{align*}
\delta^{3\ell+2}a^2 &= \delta^{3\ell-3}x^5 bxabxa^2\sim \delta^{3\ell-3}b^5 abxabx x \twist \delta^{3\ell-3}b^5 x \\&\sim \delta^{3\ell-3} \delta b^4
= \delta^{3\ell-3} a \delta b^3 \sim \delta^{3\ell-1} b^2 \sim 
\delta^{3(\ell-1)+2} a^2,
\end{align*}
where we again used two twists for `$\twist$'. Inductively this shows that $\delta^{3\ell+2}a^2$ can be untwisted with $2\ell+2$ twists as claimed. 
\end{example}

\Cref{ex:exactex1} combined with the results from \cite{BBL} for $3$-strand torus knots shows that
the equalities $\gtop = \tu =\frac{\hatsigma}{2}$ hold for all strongly quasipositive $3$-braid knots in Xu normal form (a) or (b) from \Cref{lem:xuform}, where $\hatsigma = \left|\sigma \right|$ except for certain torus knots of braid index $3$; see also \Cref{rem:sign}. We next consider a sub-case of case (c) from \Cref{lem:xuform}. 

\begin{example}\label{ex:exactex2}
Let $K$ be a knot that is the closure of a $3$-braid in Xu normal form $\beta = \delta^{3\ell +1} a^{u_1}b^{u_2}$ for $\ell \geq 0, u_1, u_2 \geq 1$. Note that $u_1$ and $u_2$ must both be even for $K$ to be a knot. We claim that
\begin{align*}
\gtop(K) = \tu(K) =\frac{|\sigma(K)|}{2}=  \frac{u_1+u_2}{2} +2\ell.
\end{align*}
The proof works as in \Cref{ex:exactex1}. After $\frac{u_1+u_2-4}{2}$ positive to negative crossing changes in $\beta$ we obtain the braid $\delta^{3\ell +1}a^2 b^2$, which we can untwist with $2\ell +2$ twists as follows. For $\ell = 0$, the braid $\delta a^2 b^2$ turns into $\delta$ by two crossing changes. For $\ell = 1$, we have 
\begin{align*}
\delta^4 a^2b^2 &= 
\delta^3 xax ab^2 
= \delta^2 bx^2 b x ab^2 
= \delta ab^3 a b x ab^2 
= xa^4 x a b x ab^2 \\&\sim ab^4abxabxx \twist  ab^4x,
\end{align*}
which can be untwisted using two crossing changes. 
For $\ell \geq 2$, we have
\begin{align*}
\delta^{3\ell +1}a^2 b^2 &=\delta^{3\ell-3} xa^4 x a b x ab^2 \sim \delta^{3\ell-3} ab^4abxabxx \twist\delta^{3\ell-3} ab^4x \\
&= \delta^{3\ell-4} x a^3 ba bx  
= \delta^{3\ell-5} b x^3 a^2 x a bx 
= \delta^{3\ell-6} a b^3 x^3 b x a bx 
 \\&\sim \delta^{3\ell-6} a^3 b^3 abxabx \twist \delta^{3\ell-6} a^3 b^3 \sim \delta^{3\ell-5} a^2 b^2 = \delta^{3(\ell-2)+1} a^2 b^2,
\end{align*}
which we can untwist inductively using the two base cases above. 
\end{example}

The following proposition improves the statement from \Cref{thm:2} for braid positive $3$-braid knots under the additional assumption $u_i \geq 2$ for the exponents in the Xu normal form of their braid representatives.
In fact, we can determine $\gtop(K)$ in this case up to an error of $1$, using $\frac12 |\sigma(K)|$ as a lower bound. 

\begin{proposition}\label{prop:braidposbound}
Let $K$ be a knot that is the closure of a $3$-braid in Xu normal form
\begin{align*}
\delta^n \tau_1^{u_1} \tau_2^{u_2} \dots \tau_t^{u_t}
\quad\text{for}\quad t\geq 1,\ n \geq \frac{t}{2},\ u_1, \dots, u_t \geq 2.
\end{align*}
Then $K$ is a braid positive knot and 
\begin{align}\label{eq:braidposbound}
\frac{n+t}{3}-1 = g(K) - \frac{|\sigma(K)|}{2} \geq g(K) - \gtop(K) \geq 
\frac{n+t}{3}-2.
\end{align}
\end{proposition}

\begin{proof}
Set $U = u_1 + \cdots + u_t$.
\Cref{prop:positivity} implies that $K$ is braid positive. Moreover, we have $\frac{|\sigma(K)|}{2} = \frac{U}{2}+\frac{2n}{3}-\frac{t}{3}$ by \Cref{prop:signature} and $g(K) = \frac{U}{2}+n-1$ by \eqref{eq:genus}. Using $\frac12 |\sigma(K)| \leq \gtop(K)$, 
it remains to show that $\gtop(K) \leq \frac{|\sigma(K)|}{2} + 1$.

We distinguish two cases depending on the parity of $t$. First, let $t = 2r$ be even for $r \geq 1$. 
The conditions $2n \geq t$ and $n + t \equiv 0\pmod{3}$ imply that we can write $n = 3\ell + r$ for $\ell \geq 0$, and $K$ is the closure of
\begin{align*}
\beta = \delta^{3\ell + r} \tau_1^{u_1} \tau_2^{u_2} \dots \tau_{2r}^{u_{2r}}. 
\end{align*}
The case $r=1$ ($t=2$) is covered by \Cref{ex:exactex2}, so we can further assume that $r \geq 2$. 
There is a smooth cobordism of Euler characteristic $4r-U-4$ from $K$ to the knot that is the closure of 
\begin{align*}
\beta^\prime = \tau_{1-r}\delta^{3\ell + r-1} \prod_{i=1}^{r-2} \tau_i^2 \tau_{r-1} \tau_{r} \tau_{r+1} \prod_{i={r+2}}^{2r} \tau_i^2.
\end{align*}
Indeed, we can use $U-4r+3$ saddle moves to replace all but three of the exponents $u_i$ by $2$ and the other three by $1$. We use a last saddle move to replace $\delta$ by $\tau_{1-r}$. 
We will prove by induction on $r$ that $\beta^\prime$ turns into $\delta^{3\ell+1}$ by $2r-2$ twists.
Since the closure of $\delta^{3\ell+1}$ is the torus knot $T(3,3\ell+1)$, this will imply
\begin{align}\label{eq:claim1}
\tu\left(\cl\left(\beta^\prime\right)\right) \leq \tu(T(3,3\ell+1)) +2r-2 =
\begin{cases}
  2\ell+2r-1  & \text{if } \ell \geq 1,  \\
  2r-2 & \text{if } \ell = 0,
\end{cases}
\end{align}
where the equality follows from \cite[Lemma~5 and Theorem~1]{BBL}. 
Recall that $\delta = \tau_{i+1} \tau_{i}$, $\tau_i \delta = \delta \tau_{i+1}$ (see \eqref{eq:xurules}) and $\tau_i = \tau_{i+3m}$ for all $m \in \Z, i \in \Z$. For $r=2$, we thus have $\tau_{1-r} = \tau_2$ 
and 
\begin{align*}
\beta^\prime &= \tau_{2}\delta^{3\ell + 1}  \tau_{1} \tau_{2} \tau_{3}  \tau_4^2 
= \delta^{3\ell} \tau_2 \tau_1 \tau_0 \tau_{1} \tau_{2} \tau_{3}  \tau_4^2 
= \delta^{3\ell} \tau_0 \tau_{-1} \tau_0 \tau_{1} \tau_{2} \tau_{3}  \tau_4^2 
\\&\sim \delta^{3\ell} \tau_2 \tau_{1} \tau_2 \tau_{3} \tau_{4} \tau_{5}  \tau_6^2  
\twist \delta^{3\ell} \tau_2  \tau_6 \sim 
\delta^{3\ell+1},
\end{align*}
so $\beta^\prime$ indeed turns into $\delta^{3\ell+1}$ using $2r-2=2$ twists in this case.
Now, for $r \geq 3$, consider
\begin{align*}
\beta^\prime 
&= \tau_{1-r}\delta^{3\ell + r-2} \prod_{i=1}^{r-3} \tau_{i-1}^2 \delta \tau_{r-2}^2\tau_{r-1} \tau_{r} \tau_{r+1} \tau_{r+2}^2\prod_{i={r+3}}^{2r} \tau_i^2
\\&= \tau_{1-r}\delta^{3\ell + r-2} \prod_{i=1}^{r-3} \tau_{i-1}^2 \tau_{r-3}\tau_{r-2} \tau_{r-3} \tau_{r-2}\tau_{r-1} \tau_{r} \tau_{r+1} \tau_{r+2}^2\prod_{i={r+3}}^{2r} \tau_i^2
\\&\sim \tau_{(1-r)-(r-4)}\delta^{3\ell + r-2} \prod_{i=1}^{r-3} \tau_{i-1-(r-4)}^2 \tau_{1}\tau_{2} \tau_{1} \tau_{2}\tau_{3} \tau_{4} \tau_{5} \tau_{6}^2\prod_{i={r+3}}^{2r} \tau_{i-(r-4)}^2
\\&\twist \tau_{(1-r)-r+1}
\delta^{3\ell + r-2} \prod_{i=1}^{r-3} \tau_{i-r}^2 \tau_{1}\tau_{2}  \tau_{3}\prod_{i={r+3}}^{2r} \tau_{i-r+1}^2
\\&\sim 
\tau_{(1-r)+1}\delta^{3\ell + r-2} \prod_{i=1}^{r-3} \tau_{i}^2 \tau_{r-2}\tau_{r-1}  \tau_{r}\prod_{i={r+1}}^{2(r-1)} \tau_{i}^2.
\end{align*}
The braid $\beta^\prime$ hence turns into $\tau_{(1-r)+1}\delta^{3\ell + r-2} \prod_{i=1}^{r-3} \tau_{i}^2 \tau_{r-2}\tau_{r-1}  \tau_{r}\prod_{i={r+1}}^{2(r-1)} \tau_{i}^2$ by two twists and inductively we get that $\beta^\prime$ turns into
\begin{align*}
\tau_{(1-r)+r-2}\delta^{3\ell + 1} \tau_{1}\tau_{2}  \tau_{3}\tau_{4}^2
\end{align*}
by $2(r-2)$ twists. Since $\tau_{(1-r)+r-2} = \tau_2$, this braid is the same as the one from the base case $r=2$ and therefore can be untwisted with two twists. We obtain that $\beta^\prime$ becomes $\delta^{3\ell+1}$ by $2r-2$ twists.
\Cref{eq:claim1} follows and we get
\begin{align*}
\gtop(K) &\leq \gtop \left(\cl\left(\beta^\prime\right)\right) 
+ \frac{U-4r+4}{2} 
\leq \tu \left(\cl\left(\beta^\prime\right)\right)+
\frac{U}{2} -2r+2   
\\&\leq 
\begin{cases}
  \frac{U}{2} +2\ell+1 =  \frac{|\sigma(K)|}{2} + 1  & \text{if } \ell \geq 1,  \\
  \frac{U}{2} = \frac{|\sigma(K)|}{2} & \text{if } \ell = 0.
\end{cases}
\end{align*}

Next, let $t = 2r+1$ be odd for $r \geq 0$. The conditions $2n \geq t$ and $n + t \equiv 0\pmod{3}$ imply that we can write $n = 3\ell + r+2$ for $\ell \geq 0$, and $K$ is the closure of
\begin{align*}
\beta = \delta^{3\ell + r+2} \tau_1^{u_1} \tau_2^{u_2} \dots \tau_{2r+1}^{u_{2r+1}}. 
\end{align*}
The case $t = 1$ is covered by \Cref{ex:exactex1}, so we can further assume that $r \geq 1$. 
There is a smooth cobordism of Euler characteristic $4r-U-2$ from $K$ to the knot that is the closure of 
\begin{align*}
\beta^\prime = \tau_{1-r}\delta^{3\ell + r+1} \prod_{i=1}^{r-1} \tau_i^2 \tau_{r} \tau_{r+1} \tau_{r+2} \prod_{i={r+3}}^{2r+1} \tau_i^2,
\end{align*}
similar to the cobordism considered in the above case. 
We prove by induction on $r$ that $\beta^\prime$ turns into $\delta^{3(\ell+1)+1}$ by $2r-2$ twists, hence 
\begin{align*}
 \tu\left(\cl\left(\beta^\prime\right)\right) \leq \tu(T(3,3(\ell+1)+1)) +2r-2 
 = 2\ell+2r+1.
\end{align*}
For $r=1$, we have 
\begin{align}\label{eq:caser=1}
\beta^\prime = \tau_{0}\delta^{3\ell + 2}  \tau_{1} \tau_{2} \tau_{3} 
= \delta^{3\ell+3} \tau_{2} \tau_{3} \sim \delta^{3(\ell+1)+1}.
\end{align}
For $r=2$, we have
\begin{align*}
\beta^\prime &= \tau_{2}\delta^{3\ell + 3}  \tau_1^2 \tau_{2} \tau_{3} \tau_{4}  \tau_5^2
= \tau_{2}\delta^{3\ell + 2}  \tau_0\tau_1 \tau_0\tau_1 \tau_{2} \tau_{3} \tau_{4}  \tau_5^2 \\&\sim \tau_{0}\delta^{3\ell + 2}  \tau_1\tau_2 \tau_1\tau_2 \tau_{3} \tau_{4} \tau_{5}  \tau_6^2
\twist \tau_{0}\delta^{3\ell + 2}  \tau_1\tau_2 \tau_6 \sim \delta^{3(\ell+1)+1}
\end{align*}
using \Cref{eq:caser=1} in the last step.
Now, for $r \geq 3$, consider
\begin{align*}
\beta^\prime 
&= \tau_{1-r}\delta^{3\ell + r} \prod_{i=1}^{r-2} \tau_{i-1}^2 \delta \tau_{r-1}^2\tau_{r} \tau_{r+1} \tau_{r+2} \tau_{r+3}^2\prod_{i={r+4}}^{2r+1} \tau_i^2
\\&= \tau_{1-r}\delta^{3\ell + r} \prod_{i=1}^{r-2} \tau_{i-1}^2 \tau_{r-2}\tau_{r-1}\tau_{r-2} \tau_{r-1}\tau_{r} \tau_{r+1} \tau_{r+2} \tau_{r+3}^2\prod_{i={r+4}}^{2r+1} \tau_i^2
\\&\sim \tau_{(1-r)-(r-3)}\delta^{3\ell + r} \prod_{i=1}^{r-2} \tau_{i-1-(r-3)}^2 \tau_{1}\tau_{2}\tau_{1} \tau_{2}\tau_{3} \tau_{4} \tau_{5} \tau_{6}^2\prod_{i={r+4}}^{2r+1} \tau_{i-(r-3)}^2
\\&\twist \tau_{(1-r)-r}\delta^{3\ell + r} \prod_{i=1}^{r-2} \tau_{i-r-1}^2 \tau_{1}\tau_{2}\tau_{3}\prod_{i={r+4}}^{2r+1} \tau_{i-r}^2
\\&\sim \tau_{(1-r)+1}\delta^{3\ell + r} \prod_{i=1}^{r-2} \tau_{i}^2 \tau_{r-1}\tau_{r}\tau_{r+1}\prod_{i={r+2}}^{2(r-1)+1} \tau_{i}^2.
\end{align*}
Inductively we get that $\beta^\prime$ turns into
\begin{align*}
\tau_{(1-r)+r-2}\delta^{3\ell + 3} \tau_{1}^2\tau_{2}  \tau_{3}\tau_{4}\tau_5^2
=\tau_{2}\delta^{3\ell + 3} \tau_{1}^2\tau_{2}  \tau_{3}\tau_{4}\tau_5^2
\end{align*}
by $2(r-2)$ twists, so into $\delta^{3(\ell+1)+1}$ by $2r-2$ twists. We obtain
\begin{align*}
\gtop(K) &\leq \gtop \left(\cl\left(\beta^\prime\right)\right)+ \frac{U-4r+2}{2} 
\leq \tu \left(\cl\left(\beta^\prime\right)\right)+
\frac{U}{2} -2r+1  
\\&\leq
  \frac{U}{2} +2\ell+2 = \frac{|\sigma(K)|}{2} + 1.\qedhere
\end{align*}
\end{proof}

\begin{remark}\label{rem:exactfamilies}
The proof of \Cref{prop:braidposbound} (more precisely, the first case with~$\ell = 0$) shows that the first inequality in \eqref{eq:braidposbound} is an equality when~$2n=t$.
\end{remark}

\begin{example}\label{ex:galg}
For knots $K$ of low Seifert genus arising as closure of positive 3-braids,
we can often apply the untwisting moves shown in \cref{sec:proofs}
and show $\frac12 \widehat{\sigma}(K) = \tu(K)$, thus determining $\gtop(K)$.
Here is a selection of positive 3-braids
that close off to knots of Seifert genus 6 and 7,
for which that strategy did not succeed:
\begin{align*}
\delta^3 a^2b^2xabx &\sim a^3b^3a^2b^2a^2b^2\\
\delta^4 a^2 bxab &\sim \Delta a^3b^2a^2b^2a^2\\
\delta^4a^4bxab &\sim \Delta a^5b^2a^2b^2a^2\\
\delta^4a^2b^2xa^2b &\sim \Delta a^3b^3a^2b^3a^2\\
\delta^6 a^2bx &\sim \Delta^3a^3b^2a^2.
\end{align*}
First, we note that for all of these knots $K$, we have the lower bound $\frac12 \hatsigma(K) = \frac12|\sigma(K)| = g(K)-2 \leq \gtop(K)$.
Second, in the search for upper bounds, we are able to find a knot $J$ such that $K\leadsto J$ and $\tu(J) = g(J) - 1$, thus proving $\gtop(K) \leq g(K) - 1$, for each of these knots $K$.  However, we are unable to find $J$ with $K\leadsto J$ and $\tu(J) = g(J) - 2$.
Nevertheless, we can prove $\gtop(K) = g(K)-2$ for all of these knots $K$ in a different way,
which is practical for individual knots with low Seifert genus. 
Namely, a computer search~\cite{Lew} reveals that the \emph{algebraic genus} $g_{\text{alg}}(K)$, which is defined in terms of Seifert matrices of $K$ and provides 
an upper bound $\gtop(K) \leq g_{\text{alg}}(K)$~%
\cite{FL}, satisfies $g_{\text{alg}}(K) \leq g(K) - 2$ for all our knots $K$.
Yet it remains an open question whether $\tu(K) = g(K) - 1$ or $\tu(K) = g(K) -2$
for those knots $K$.
\end{example}

\begin{remark}\label{rem:sign}
The maximal Levine--Tristram signature, see \eqref{eqn:sigmahat} in \cref{sec:proofs}, provides a good computable lower bound for $\gtop$:
\begin{align*}
\widehat{\sigma}(K) = \underset{\omega \in S^1 \backslash \Delta_K^{-1}(0)}{\operatorname{max}} \left \vert \sigma_\omega(K) \right \vert \leq 2\gtop(K).
\end{align*}
The function $S^1 \to \Z$, $\omega\mapsto \sigma_{\omega}(K)$,
is piecewise constant and jumps only at zeroes of the Alexander polynomial $\Delta_K$. A priori, its maximum absolute value may be assumed anywhere on $S^1\setminus \{1\}$.

\begin{figure}[t]
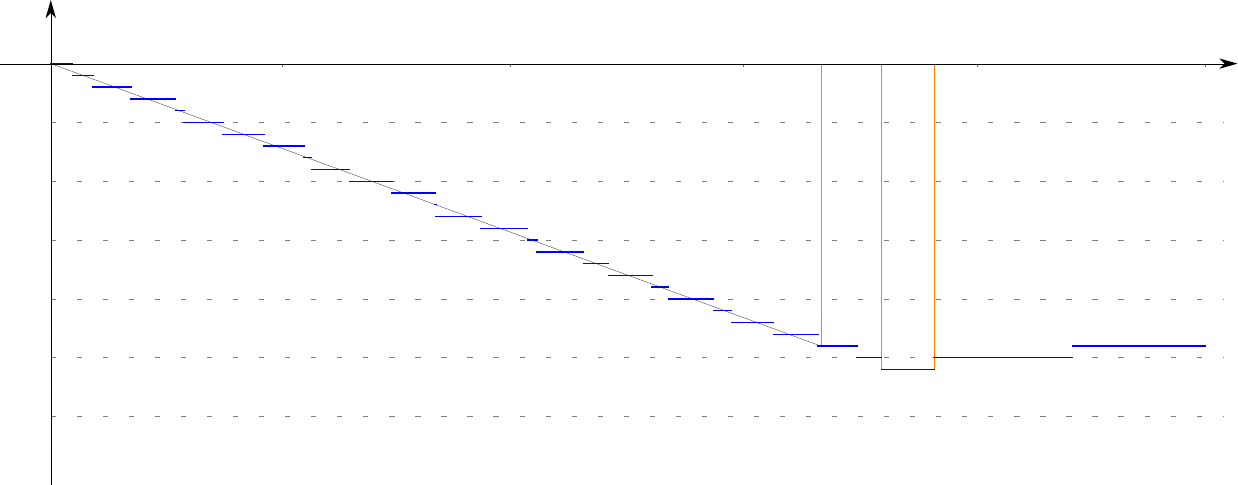
\caption{In blue, the graph of the Levine--Tristram signature $\sigma_{e^{2\pi i t}}(K)$
for $t \in \left[0,\frac12\right]$ and $K$ the closure of the 3-braid
$\left(a^2b^2\right)^8\left(a^5b^5\right)^4 \sim \delta^{12}\left(a^4b^4x^4\right)^2a^4b^4(xab)^5x$.
In black, the linear approximation by \cite[Corollary~4.4]{GG}
for $t \in [0,\frac13]$.
The maximum absolute value $\widehat{\sigma}(K)$
of $\sigma_{e^{2\pi i t}}(K)$,
which equals $|\sigma(K)| + 4 = |\sigma_{\zeta3}(K)| + 4$,
is assumed between $0.3599$ and $0.3826$ (rounded).
This graph was drawn with sage~\cite{Sage}, by computing
the signatures
of the Hermitian matrices $(1-\omega)A + (1-\overline{\omega})A^{\top}$,
for $A$ a Seifert matrix of $K$, and $\omega$
between the roots of $\Delta_K$ on $S^1$.}
\label{fig:sigprofile}
\end{figure}
In the above examples and \Cref{prop:braidposbound} 
we have seen that for certain families of $3$-braid knots, $\widehat{\sigma} = |\sigma| = 2\gtop$,
where $\sigma = \sigma_{-1} = \sigma_{e^{\pi i}}$ is the classical knot signature.
This also holds for the $T(3,3k+m)$ torus knots with $m\in\{1,2\}$ and odd $k\geq 1$ \cite{BBL}.
For even~$k$ on the other hand, e.g.~for $T(3,7)$,
one finds $\widehat{\sigma } = |\sigma_{\omega}| = |\sigma| + 2$
for $\omega$ chosen only one jump-point of the Levine--Tristram signature away from $e^{\pi i}$,
i.e.~$\omega = e^{2 \pi i t}$ for
\[ t \in \left(\frac12 - \frac5{18k + 6m},\ \frac12 - \frac1{18k + 6m}\right). \]
This observation relies on the fact that the jumps of the Levine--Tristram signatures of torus knots
are well understood~\cite{Lit, Ban}.
We have also seen examples where $\hatsigma = |\sigma_{\zeta3}|$, namely the closures of $(abx)^{2k} abx^2abx^2$ for $k \geq 0$; see \cref{rmk:abxexact}.
Overall, whenever we could precisely determine the topological 4-genus of a 3-braid knot $K$,
then the maximum absolute value of $\sigma_{e^{2 \pi i t}}$ was either assumed at $t = \frac{1}{3}$,
or at $t = \frac{1}{2}$, or close to $t = \frac{1}{2}$.

However, ${\widehat{\sigma}(K) - \max\left(|\sigma_{\zeta3}(K)|, |\sigma(K)|\right)}$ is unbounded for $K$ ranging over positive $3$-braid knots.
Indeed, let $K_n$ be the closure of $\left(a^2b^2\right)^{2n} \left(a^5b^5\right)^n$, which is a knot if $n$ is not a multiple of 3. The Levine--Tristram signatures of $K_4$ are shown in \cref{fig:sigprofile}.
Using \cref{prop:garsidesig} and 
\cite[Corollary~4.4]{GG} (as in \eqref{eqn:sigmahat}), we find
\begin{equation}\label{eq:sigandzeta3}
\sigma(K_n) = -12n, \quad \sigma_{\zeta3}(K_n) \approx -12n,
\end{equation}
with an error of at most $2$. Let us now estimate $\widehat{\sigma}(K_n)$.
We use the fact
that $K_n$ can be transformed by $n+10$ saddle moves into the
link $J_n = T(2,10)^{\#(n-1)} \# T(3,6n) \# T(2,-4n)$.
The latter satisfies, up to a constant error (i.e.~an error that is independent of $n$),
\[
\sigma_{e^{6 \pi i/7}}(J_n) \approx
-9n - 8n + \frac67\cdot 4n = -(13+\frac{4}{7})n,
\]
which implies that (again with constant error)
\[
\widehat{\sigma}(K_n) \geq |\sigma_{e^{6 \pi i/7}}(K_n)|
\geq |\sigma_{e^{6 \pi i/7}}(J_n)| - n - 10 \approx (12+\frac{4}{7})n.
\]
This estimate, combined with \eqref{eq:sigandzeta3}, shows that
\[
\lim_{n\to\infty} \widehat{\sigma}(K_n) - \max\left(|\sigma_{\zeta3}(K_n)|, |\sigma(K_n)|\right) = +\infty.
\]

This last example shows that determining $\widehat{\sigma}$ and $\gtop$ for all 3-braid knots, or even just closures of positive 3-braids, could be a hard problem.
\end{remark}

\bibliographystyle{alpha}

\end{document}